\documentclass{article}
\usepackage[utf8]{inputenc}
\usepackage{amsthm, amssymb, amsmath}
\usepackage{xcolor}
\usepackage{authblk, verbatim}

\title{Conflict-Free Coloring of Star-Free Graphs on Open Neighborhoods}
\author{Sriram Bhyravarapu}
\author{Subrahmanyam Kalyanasundaram}
\author{Rogers Mathew}
\affil{Department of Computer Science and Engineering, \authorcr
Indian Institute of Technology Hyderabad, India - 502285. \authorcr {\tt \{cs16resch11001, subruk, rogers\}@iith.ac.in}
}

\theoremstyle{definition}

\theoremstyle{plain}
\newtheorem{theorem}{Theorem}
\newtheorem{lemma}[theorem]{Lemma}
\newtheorem{corollary}[theorem]{Corollary}

\newtheorem{definition}[theorem]{Definition}

\theoremstyle{remark}

\newcommand{\cfon}{CFON}
\newcommand{\chion}{\chi_{ON}(G)}
\newcommand{\cf}{CFON}

\begin{document}

\maketitle

\begin{abstract}
Given a graph, the conflict-free coloring problem on open neighborhoods (CFON)
asks to color the vertices of the graph so that all the vertices have a
uniquely colored vertex in its open neighborhood. The smallest number of 
colors required for such a coloring is called the conflict-free chromatic number
and denoted $\chion{}$.
In this note, we study this problem on 
$S_k$-free graphs where $S_k$ is a star on $k+1$ vertices. 
When $G$ is $S_k$-free, we show that  
  $\chion = O(k\cdot \log^{2+\epsilon}\Delta)$, for any $\epsilon > 0$, 
  where $\Delta$ denotes the maximum degree of $G$.
Further, we show existence of claw-free ($S_3$-free) graphs that require $\Omega(\log \Delta)$ colors. 
\end{abstract}

A \emph{conflict-free coloring} of a hypergraph $H=(V,E)$, denoted by $\chi_{CF}(H)$, is an assignment of colors to the points in $V$ such that every $e \in E$ contains a point whose color is distinct from that of every other point in $e$.  
The \emph{conflict-free chromatic number} of $H$, denoted $\chi_{CF}(H)$, is the smallest 
number of colors required for such a coloring.
Introduced \cite{Even2002}  in 2002 by Even, Lotker, Ron and Smorodinsky, many variants of the problem
have been extensively studied \cite{smorosurvey}.
Conflict-free colorings have also been studied in the context of hypergraphs created out of graphs. One such popular variant is  conflict-free coloring with respect to open neighborhoods (or CFON coloring) in a graph. 
Given a graph $G$, for any vertex $v \in V(G)$, let $N_G(v) : = \{u \in V(G)~:~\{u,v\} \in E(G)\}$ denote the \emph{open neighborhood} of $v$ in $G$. 
\begin{definition}[CFON Coloring of Graphs] \label{defn:open_CF}
Given a graph $G$, let $H$ be  the hypergraph with $V(H) = V(G)$ and $E(H) = \{N_G(v):v \in V(G)\}$. 
A conflict-free coloring on open neighborhoods (CFON coloring) of $G$ is defined
as a conflict-free coloring of $H$.
The CFON chromatic number of $G$, denoted by $\chi_{ON}(G)$, is equal to $\chi_{CF}(H)$. 
\end{definition} 




\begin{definition}[Star and Claw]
The complete bipartite graph $K_{1,k}$ is referred to as \emph{star on $k+1$ vertices} and denoted by $S_k$.
The graph $S_3$ is also known by the name \emph{claw}.
\end{definition}
A graph is said to be \emph{$S_k$-free} (\emph{claw-free}) if 
it does not contain an $S_k$ ($S_3$) as an induced subgraph. 
In this paper, we study the \cfon{} problem on $S_k$-free graphs. For a graph $G$ with maximum degree $\Delta$,  it is known that $\chion \leq \Delta + 1$ and this bound is tight in general. 
We improve this result for $S_k$-free graphs 
for most values of $k$.
\begin{theorem}
\label{thm:S_k-free}
Let $G$ be an $S_k$-free graph with maximum degree $\Delta$. Then, $\chion = O(k\log^{2+\epsilon}\Delta)$, for any $\epsilon > 0$.  
\end{theorem}

The following Theorems \ref{thm:pach_tardos_1} and \ref{thm:pach_main} from \cite{Pach2009} will be used in the proof of Theorem \ref{thm:S_k-free}.

\begin{theorem}[Theorem 1.1 in \cite{Pach2009}]
\label{thm:pach_tardos_1}
Let $H$ be a hypergraph and let $\Delta$ be the maximum degree of a vertex in $H$. Then the conflict-free chromatic number of $H$ is  at most $\Delta + 1$. 
This bound is optimal and the corresponding coloring can be found in linear deterministic time. 
\end{theorem}

Let $G$ be a graph with maximum degree $\Delta$. 
The above theorem 
implies that 
$\chion \leq \Delta + 1$. 
The subdivided clique $K_n^*$ has a maximum degree of $n-1$ and satisfies $\chi_{ON}(K_n^*)=n$.  
This serves as a tight example to the above bound.

\begin{theorem}[Theorem 1.2 in \cite{Pach2009}]
\label{thm:pach_main}
For any positive integers $t$ and $\Gamma$, the conflict-free chromatic number of any hypergraph in which each edge is of size at least $2t-1$ and each edge intersects at most $\Gamma$ others is $O(t\Gamma^{1/t}\log \Gamma)$. There is a randomized polynomial time algorithm to find such a coloring.   
\end{theorem}

 We begin with an auxiliary lemma. 

\begin{lemma}
\label{lem:aux_star_free}
Let $G$ be an $S_k$-free graph with no isolated vertices. Let $A \subseteq V(G)$ be an independent set of vertices in $G$. Let $B = V(G) \setminus A$. There is a way to color the vertices in $B$ using at most $k$ colors such that every vertex in $A$ sees some color appear exactly once in its open neighborhood. 
\end{lemma}
\begin{proof}
Construct a hypergraph $H=(V,E)$ with $V(H) = B$ and $E(H) = \{N_G(v) \cap B~:~v \in A\}$. Since $G$ is $S_k$-free, no element of $V(H)$ is present in more than $k-1$ hyperedges. From Theorem \ref{thm:pach_tardos_1}, we  have a conflict free coloring of $H$ using $k$ colors. 
\end{proof}

\begin{proof}[Proof of Theorem \ref{thm:S_k-free}]
We use an iterative process to color the vertices. 
Consider an $S_k$-free graph $G = (V, E)$ with maximum degree $\Delta = \Delta(G)$.
We first partition $V = V_0$ into $U_1$ and $V_1$, where $U_1 = \{ 
v \in V_0 : \deg_{G[V_0]}(v) > \log \Delta\}$, and $V_1 = V_0 \setminus U_1$. We construct a hypergraph $H_1$ from $G$ with 
$V(H_1) = V_0$ and $E(H_1) = \{N_{G}(v) \cap V_0 : v\in U_1\}$.
Every hyperedge $e \in E(H_1)$ satisfies $|e| \geq \log \Delta + 1$, and $e$ intersects at most $\Delta^2$ other hyperedges in $H_1$.  
Applying Theorem \ref{thm:pach_main} with $t=\frac{\log\Delta}{2}+1$ and $\Gamma = \Delta^2$, we get the conflict-free chromatic number of $H_1$ to be at most $\alpha(\log\Delta)^2$, where $\alpha>0$ is some constant. 
That means, there is an assignment $C_1: V_0 \rightarrow [\alpha(\log\Delta)^2]$ such that every vertex in $U_1$  sees
some color exactly once in its open neighborhood.

Now notice that $V_1$ is the set of all vertices that have degree
at most $\log \Delta$. We repeat the above process by setting
$U_2 =  \{ v \in V_1 : \deg_{G[V_1]}(v) > \log \Delta(G[V_1]) \geq \log \log \Delta\}$, and $V_2 = V_1 \setminus U_2$.
We construct a hypergraph $H_2$ with $V(H_2) = V_1$ and 
$E(H_2) = \{N_G(v) \cap V_1:v\in U_2\}$. 
Every hyperedge in $H_2$ is of size at least $\log\log\Delta + 1$ and each hyperedge 
intersects at most $\log^2\Delta$ other hyperedges. Applying Theorem \ref{thm:pach_main} with $t = \frac{\log\log\Delta}{2} + 1$ and $\Gamma = \log^2\Delta$,  we get 
an assignment $C_2: V_1 \rightarrow [\alpha(\log\log\Delta)^2]$ such that every vertex in $U_2$ sees
some color exactly once in its open neighborhood.

We iterate in this manner till we get, say $V_r$, which is an independent
set in $G$.
We now use Lemma \ref{lem:aux_star_free} to assign 
$C_{r+1}: V \rightarrow [k]$ so that every vertex in $V_r$
sees a color exactly once in its open neighborhood.

Now consider the color assignment $C$ formed by the Cartesian product
of the previous\footnote{Values $C_i(v)$ that are not assigned
are notionally set to 0.} 
assignments. That is $C(v) =(C_1(v), C_2(v), \ldots, C_r(v), C_{r+1}(v))$.
Notice that $C$ is a CFON coloring. This is because $V = U_1 \cup U_2 \cup \ldots \cup U_{r} \cup V_r$ is a partition of $V$. 
If $v\in U_i$, then $v$ has a neighbor that is uniquely colored by 
the assignment $C_i$. 
Also, every $v \in V_r$ has a 
neighbor that is uniquely colored by 
the assignment $C_{r+1}$. 
The number of colors used is 
$$ (\alpha (\log \Delta)^2) \cdot (\alpha (\log \log \Delta)^2) \cdot \dots \cdot (\alpha (\underbrace{\log \log \ldots \log}_{r \text{ times}} \Delta)^2) \cdot k\;,$$
which is upper bounded by $k \log^{2 + \epsilon} \Delta$.
This follows by noting that $r \leq \log^* \Delta$, the iterated logarithm of $\Delta$. Thus we have $\chion{} = O(k \log^{2 + \epsilon} \Delta)$.
\end{proof}

\noindent\textbf{Algorithmic note:} It is easy to see that the construction of sets $V_i$ and $U_i$ can be done in deterministic polynomial time. Theorem \ref{thm:pach_tardos_1}
states that the coloring $C_{r+1}$ 
can be computed in
deterministic linear time.
What is left is to know whether the colorings $C_i$ ($1 \leq i \leq r$) can be computed in deterministic polynomial time. Theorem \ref{thm:pach_main} states that the colorings $C_i$ ($1 \leq i \leq r$) can be obtained in 
randomized polynomial time. In the proof of Theorem \ref{thm:pach_main} in \cite{Pach2009}, an algorithmic version of the Local Lemma is used to obtain a randomized algorithm for finding the desired coloring for the hypergraph under consideration. There are  deterministic algorithms known for the Local Lemma \cite{lllcgh,harris2019deterministic} which can be used in place
of the randomized algorithm used in \cite{Pach2009}. 
By applying Theorem 1.1 (1) from \cite{harris2019deterministic}, 
we get a deterministic polynomial time algorithm to find 
 the colorings $C_i$ ($1 \leq i \leq r$).
However, the deterministic version of Local Lemma causes us to
use $O(t\Gamma^{(1+\delta)/{t}}\log \Gamma)$ colors for a 
constant $\delta>0$. This is slightly worse than the bound in 
Theorem \ref{thm:pach_main}. However, 
this weaker bound suffices to get a conflict-free coloring of the hypergraphs $H_0, H_1, \ldots$ using asymptotically the same number of colors as before. We thus have a deterministic polynomial time algorithm for \cf{} coloring the vertices of an $S_k$-free graph with maximum degree $\Delta$ using $O(k\log^{2+\epsilon}\Delta)$ colors, for any $\epsilon > 0$.

Given a graph $G$, the \emph{line graph} of $G$, denoted by $L(G)$, is the graph with $V(L(G)) = E(G)$ and $E(L(G)) = \{\{e,f\} : \mbox{edges }e \mbox{ and } f \mbox{ share an endpoint in } G\}$. It is easy to see that the line graph of any graph is claw-free. In what follows, we use this fact to show the existence of claw-free graphs of high \cf{} chromatic number.
\begin{theorem}
\label{thm:claw_free_example}
There exist claw-free graphs $G$ on $n$ vertices with $\chion{}=\Omega(\log n)$. 
\end{theorem}
\begin{proof}
Let $m$ be a positive integer. Consider the complete graph $K_m$ on $m$ vertices. Let $n = {m \choose 2}$ denote the number of vertices in the line graph of $K_m$. Consider the \cf{} coloring problem for the line graph of $K_m$. In other words, we need to color the edges of $K_m$ with the minimum number of colors such that every edge sees some color exactly once in its open neighborhood. 
Consider an optimal \cf{} coloring $C:E(K_m) \rightarrow \{1,2, \ldots , k\}$ of the edges of $K_m$ that uses, say $k$ colors. Below we show that $k \geq \log m$. 

 Corresponding to each $v \in K_m$, we construct a $k$-bit $0$-$1$ vector $g(v)$. The $i$-th bit $g_i(v) = 1$ if there is exactly one edge incident on $v$ with the color $i$. Otherwise, $g_i(v) = 0$. 
 Since $C$ is a valid \cf{} coloring of the edges of $K_m$, for any two distinct vertices $u,v \in V(K_m)$, $g(u)$ should differ from $g(v)$ in at least one position. 
 Consider the edge $\{u,v\}\in E(G)$. Let $\{v,w\}\in E(G)$ be the uniquely colored edge in the open neighborhood of $\{u,v\}$ and let $C(\{v,w\})=i$. This implies that none of the other edges incident on the vertices $u$ or $v$ are assigned the color $i$ except possibly the edge $\{u,v\}$ itself.
 Thus $g(u)$ and $g(v)$ differ in at least one position. 
 This implies that $k \geq \log m$.
 \end{proof}

Since a line graph is claw-free, Theorems \ref{thm:S_k-free} and \ref{thm:claw_free_example} imply the following corollary.
\begin{corollary}
\label{cor:line_graph}
Let $G$ be the line graph of a graph. Let $\Delta$ denote the maximum degree of $G$. Then,  $\chion{}=O(\log^{2+\epsilon}\Delta)$, 
for any $\epsilon>0$. Further, there exist line graphs with maximum degree $\Delta$ having 
$\chion{}=\Omega(\log \Delta)$. 
\end{corollary}

\noindent \textbf{Note:} Very recently, D\c{e}bski and Przyby\l{}o \cite{dbski2020conflictfree}, 
in independent and simultaneous work, 
showed that the 
closed neighborhood conflict-free chromatic number or 
CFCN chromatic number (defined analogously to Definition \ref{defn:open_CF})
of line graphs is $O(\log\Delta)$. 
In Theorem 3 of \cite{dbski2020conflictfree}, it is shown that  
$\chi_{CN}(L(K_n)) = \Omega(\log n)$.
Since the CFCN chromatic number of a  graph is at most twice its \cf{} chromatic 
number, this lower bound proved in \cite{dbski2020conflictfree} implies Theorem \ref{thm:claw_free_example}.

\bibliographystyle{alpha}
\bibliography{bibfile}
\end{document}